\documentclass[a4paper,draft]{amsproc}
\usepackage{amssymb}
\usepackage[hyphens]{url} 

\theoremstyle{plain}
 \newtheorem{thm}{Theorem}[section]
 \newtheorem{prop}{Proposition}[section]
 \newtheorem{lem}{Lemma}[section]
 
\theoremstyle{definition}

\theoremstyle{remark}
 \newtheorem{rem}{Remark}[section]

 \numberwithin{equation}{section}

\renewcommand{\leq}{\leqslant}\renewcommand{\geq}{\geqslant}

\title[Evaluating the sum of positive series]{Evaluating the sum of convergent positive series}



\author[Abramov]{\bfseries Vyacheslav M. Abramov} 

\address{ 
24 Sagan Drive, \\ 
Cranbourne North   \\ 
Victoria-3977\\
Australia}
\email{vabramov126@gmail.com}

\thanks{Communicated by ...} 

\begin{document}

{\begin{flushleft}\baselineskip9pt\scriptsize
MANUSCRIPT
\end{flushleft}}
\vspace{18mm} \setcounter{page}{1} \thispagestyle{empty}

\begin{abstract}
We provide numerical procedures for possibly best evaluating the sum of positive series under quite general setting. Our procedures are based on the application of a generalized version of Kummer's test.
\end{abstract}

\maketitle

\section{Introduction}  
\subsection{Formulation of the problem and literature review}
Let
\begin{equation}\label{0}
\sum_{n=n_0}^{\infty}a_n=s
\end{equation}
be a convergent series with positive terms $a_n$, where $n_0$ is some integer that initiates the series.
The aim of this paper is to provide effective numerical procedures for evaluating $s$. This problem is old and important, and there is a number of known considerations in the literature.
A well motivation of this problem can be found in Boas \cite{B}. On page 237, Boas \cite{B} writes: ``Textbooks spend a lot of time on tests for convergence that are of little practical value, since a convergent series either converges rapidly, in which case almost all test will do; or it converges slowly, in which case it is not going to be of much use unless there is some way to get at its sum without adding up unreasonable number of terms." Then on page 238, he writes: ``It is different, but related, and somewhat more difficult, problem to calculate the sum of series when it would take an unreasonable or impossible number of terms to get it to a desired degree of accuracy. For example, $\sum_{n=2}^{\infty}n^{-1}(\log n)^{-2}$ would require about $10^{87}$ terms (the exact number is given below, on p. 240) to get its sum to 2 decimal places, but the sum is known, by indirect methods, to be approximately $2.10974$." In \cite[page 242]{B} Boas provides a table containing the information about convergence of different series including $\sum_{n=2}^{\infty}n^{-1}(\log n)^{-2}$, $\sum_{n=3}^{\infty}n^{-1}(\log n)^{-1}(\log\log n)^{-2}$ and many other interesting series.

There are different ways of approaching this problem in the literature. Let $f(x)$ be a continuous, positive and decreasing function, and $f(n)=a_n$.
Denote $R_n=\sum_{k=n}^{\infty}a_n$. Then $\int_{n+1}^{\infty}f(x)\leq R_n\leq \int_{n}^{\infty}f(x)$.
Morley \cite{M} showed that if $f(x)$ is also convex, then
\[
\int_{n}^{\infty}f(x)\mathrm{d}x-\frac{1}{2}f(n)\leq R_n\leq \int_{n}^{\infty}f(x)\mathrm{d}x-\frac{1}{2}f(n+1).
\]
Under the same assumption, this result was further sharpened by Nelsen \cite{N} to the following estimate
\[
\frac{1}{2}f(n+1)+\int_{n+1}^{\infty}f(x)\leq R_n\leq \int_{n+1/2}^{\infty}f(x)\mathrm{d}x.
\]
If $f(x)$ is smooth,
Boas \cite[pages 238, 239]{B} derives the simple inequality that follows directly from Euler-Maclaurin formula and second mean-value theorem:
\[
\int_{n+1/2}^{\infty}f(x)\mathrm{d}x+\frac{1}{8}f^\prime\left(n+\frac{1}{2}\right)<R_n<\int_{n+1/2}^{\infty}f(x)\mathrm{d}x,
\]
supporting it with many examples. Braden \cite{Br} has built error bounding pair for a series considering three different tests such as integral test, limit comparison test and ratio test. That error bounding pair enables us to evaluate the total number of terms in the partial sum of the series in order to reach the necessary accuracy.

The solution of ten challenging problems of numerical analysis that include series summation as a part has been provided in \cite{BLWW}. The approach in that book concerns all major techniques of modern numerical analysis that includes matrix computation, iterative linear methods, limit extrapolation and convergence acceleration, numerical quadrature, contour integration, discretization of PDEs, global optimization, Monte Carlo and evolutionary algorithms, error control, interval and high-precision arithmetic, and many more. The problems related to series summation can be found in \cite[Appendix A]{BLWW} titled \textit{Convergence Acceleration} and in \cite[Section 3]{BLWW} titled \textit{How Far Away Is Infinity?}  Specifically in \cite[Section 3]{BLWW}, the author of the section, J\"{o}rg Waldvogel, finds the $\ell^2$-norm of the infinite matrix $A$, the entries of which are $a_{1,1} = 1$, $a_{1,2} = 1/2$, $a_{2,1} = 1/3$, $a_{1,3} = 1/4$, $a_{2,2} = 1/5$, $a_{3,1} = 1/6$, and so on. The suggested methods include the analytic transform of the function of complex variable and its contour integration in order to find the required limit of the partial sums sequence for the series arising there. 

Note that the appearance of \cite{BLWW} was an accepted challenge on the announcement of L. N. Trefethen in SIAM News \cite{Tr}, who formulated ten easy-to-state but hard-to-solve problems on numerical analysis (see also the reviews of D. H. Bailey \cite{Bai} and J. M. Borwein \cite{Bor} for the additional comments). So, \cite{BLWW} is a problem-oriented book, the methods of which can be extended to wider classes of problems keeping the accuracy and computational speed (see \cite{H}). Concerning the convergent positive series, this means that the methods of \cite{BLWW} are applicable for some classes of slowly convergent series, and the limits of partial sums of those series can also be found for them.

During last years, new methods of summation of slowly convergent series have been developed in a number of papers of Milovanovi\'{c} \cite{Milov0, Milov1, Milov2, Milov3, Milov4, Milov5} and Gautschi and Milovanovi\'{c} \cite{GM}. In particular, in most recent paper \cite{Milov5} Milovanovi\'{c} used so-called summation/integration method based on transformation of the series to weighted integrals and construction of the quadrature formulas of Gaussian type for those integrals with respect to the weight functions contained in the construction of the weighted integrals. The mentioned paper \cite{Milov5} also provides a review of the known methods for summation of slowly convergent series developed at the last time.

\subsection{Motivation}
The methods suggested in \cite{B, Br, M, N}, cannot be successful in many cases, when the analytical derivation of $\int_n^\infty f(x)\mathrm{d}x$ is hard or impossible. Even in the cases when the analytical derivation of $\int_n^\infty f(x)\mathrm{d}x$ is possible but has a complex expression, the problem of finding $n$ in order to reach the necessary accuracy can be very challenging (e.g. see the discussion in \cite[page 240]{B}). As well, the methods of \cite{BLWW} or \cite{Milov5}  can be successful in a limited number of cases, for which it is possible to evaluate the limit of the sequence of partial sums of a series by using the special transforms and analytic techniques of complex analysis developed there.

In the present paper, we estimate the sum $s$ for quite general classes of convergent series. The terms $a_n$ in \eqref{0} may have a very complex form that will make impossible to use any analytic transform available in \cite{BLWW} or \cite{Milov5} to find the limit. For instance, $a_n$ can be derived from an inhomogeneous infinite system of functional or differential equations that often appears in applied areas of probability and mathematical analysis.
The method of the present paper works in quite general situation. The only general information about qualitative properties of $a_n$, such as the sequence $\{a_{n+1}/a_{n}\}$ is strictly increasing, is known. Thus, the present paper suggests a new tool for evaluating positive series with practically required accuracy, in which the explicit formula for $a_n$ is assumed to be unknown.

\subsection{Types of basic numerical procedures and approach}

We suggest two numerical procedures for possibly best evaluating $s$. One of them is called \textit{test procedure}. It enables us to check whether the remainder of the series is smaller than given $\epsilon$. Another procedure that is called \textit{search procedure} evaluates the sum of series or its remainder.
The search procedure is based on a search method that includes test procedures at each step of the search.
We shall consider two search procedures. One of them, \textit{step-forward search}, is based on the sequential test procedures consequently evaluating the remainders of the series as long as the required accuracy is not reached. The second one, \textit{modified step-forward search}, is an improved version of step-forward search with better performance that enables us to reach good accuracy for reasonable time. It turns out that modified step-forward search solves the required problem with relatively small number of iterations justified by the numerical study of the series given in the paper.

The approach of the present paper is based on the modified version of Kummer's test given by Tong \cite{T}. We generate the test's auxiliary sequence, and on the basis of the properties of that sequence we are able to arrive at the conclusion about the required accuracy for the estimate of the series sum.

Kummer's test in its original version appeared in 1835 in \cite{K}. Since its first publication it has been revised many times, and
after more than fifty years since then Stoltz provided the clearer formulation and proof that has been well-accepted and appeared in the textbooks (see e.g. \cite[page 311]{Kn}) and well-known in our days. About 30 years ago, Tong \cite{T} proved a new version of the test that characterizes convergence or divergence of any positive series in the forms of necessary and sufficient conditions.

Although Kummer's test is a more general test than many existing particular tests such as d'Alembert test, Raabe's test, Bertrand's test and Gauss's test, it is seldom applied in practical and theoretical problems, since it required an elegant construction of an auxiliary sequence. The known applications were given for new particular tests (e.g. \cite{A}) and in the theory of difference equations (e. g. \cite{GH}). Connection of Kummer's test with regular variation is given in \cite{Rehak2}. The applications of Tong's theorem \cite{T} hitherto are unknown, and this paper presents the first one.

\subsection{Comments on the numerical study} The numerical examples of this paper are relatively simple and have only an illustrative nature.
They do not pretend to be challenging problems that cannot be solved by other known methods, but enable a reader easily understand the procedures and reproduce the computations using MATLAB or another tool. Relatively simple series for illustration purposes are often used. For instance, the slowly convergent series $\sum_{n=2}^\infty n^{-1}(\log n)^{-2}$, the sum of which is known (e. g. \cite{B}) is used in a number of papers (e. g. \cite{Br}, \cite{C}) for illustration of the methods suggested there.

 We shall study numerically the following two series. The first series, $I_1=\sum_{n=1}^{\infty} n^{-3/2}\log(n+1)$, is a series with relatively slow convergence, and the second one, $I_2=\sum_{n=1}^{\infty} n^{-7/4}\log(n+1)$, is a more regularly convergent one. Both of these series can be numerically studied by the known methods proposed in \cite{BLWW},  \cite{Br} or \cite{Milov5}. For instance, with the aid of the method of \cite[Section 3]{BLWW}, it is possible to find the limit of the partial sums of $I_1$ and obtain $I_1= 4.91715 77360 18209$ (with the accuracy of fifteen digits). By the method presented in \cite[Theorem 3.2]{Milov5} with $m=10$ and $n=80$ (the notation is taken from \cite{Milov5}; $n=80$ denotes the number of nodes in the quadrature formula) we have approximately $I_1=4.91715773601820873704547032417452640168842246152187424222353$.  
 
 The method of \cite{Br} also enables us to find the required bounds for $I_1$ and $I_2$ in order to judge about the possible number of terms for the required accuracy.

 With the algorithms of the present paper, $I_1$ is calculated with the accuracy of two digits, and $I_2$ with the accuracy of four digits.

\subsection{Outline of the paper}
The rest of the paper is organized as follows.  In Section \ref{S2}, we recall the formulation of Tong's theorem \cite{T} in the form adapted to the required numerical procedures and provide its new short proof containing the important expression that is then used in the paper. The proof is based on application of Abelian and Tauberian theorems.
In Section \ref{S3}, we explain the test procedure and justify its effectiveness on examples. In Section \ref{S4}, we explain the search procedures on two numerical examples.
In Section \ref{S5}, we conclude the paper.

\section{Tong's theorem}\label{S2} In this section, we formulate and prove the only first claim of Tong's theorem related to the convergence of the series. The second claim related to divergence is not required for our further construction.
For convenience, in the formulation and proof of the theorem, the value $n_0$ in \eqref{0} is set to $0$.

\begin{thm}\label{thm1}
Series $\sum_{n=0}^{\infty}a_n$ converges if and only if there exists a positive sequence $\zeta_n$, $n=0,1,\dots$, such that $\zeta_na_n/a_{n+1}-\zeta_{n+1}=1$.
\end{thm}

\begin{rem}
The formulation of Theorem \ref{thm1} is simpler than that in \cite{T}, where the sequence $\zeta_n$ was assumed to satisfy $\zeta_na_n/a_{n+1}-\zeta_{n+1}\geq c>0$.
\end{rem}

\begin{proof}
The elementary proof given here involves a well-known Abel theorem, its inversion for positive series as well as a Hardy-Littlewood Tauberian theorem \cite{Hardy, HL}. Below we recall the formulation of that Tauberian theorem.

\begin{lem}\label{lem1} Let the series
$
\sum_{j=0}^{\infty}a_jx^j
$
converge for $|x|<1$, and suppose that there exists $\gamma\geq0$ such that
$
\lim_{x\uparrow1}(1-x)^\gamma\sum_{n=0}^{\infty}a_jx^j=A.
$
Suppose also that $a_j\geq0$. Then, as $N\to\infty$, we have
$
\sum_{j=0}^Na_j=({A}/{\Gamma(1+\gamma)})N^\gamma(1+o(1)),
$
where $\Gamma(x)$ is Euler's Gamma-function.
\end{lem}

For $|x|<1$ introduce generating functions. Denote $A(x)=\sum_{n=1}^{\infty}a_nx^n$ and $Z(x)=\sum_{n=0}^{\infty}a_n\zeta_nx^n$. We have
\begin{equation}\label{1}
a_0\zeta_0-A(x)=(1-x)Z(x).
\end{equation}
Now both necessary and sufficient conditions follow from \eqref{1}. If $\sum_{n=0}^{\infty}a_n$ converges, then according to Abel's theorem $\lim_{x\uparrow1}A(x)=s-a_0$, and
$\zeta_0$ can be chosen satisfying the condition $\zeta_0>(s-a_0)/a_0$. According to Lemma \ref{lem1}, for large $N$ we have $\sum_{n=0}^{N}a_n\zeta_n=\big((s-a_0)/a_0\big)N(1+o(1))$, and hence the required positive sequence $\zeta_n$ exists. On the other hand, the existence of a positive sequence $z_n=a_n\zeta_n$ satisfying $\lim_{x\uparrow1}(1-x)\sum_{n=0}^{\infty}a_n\zeta_nx^n=c>0$ implies that the left-hand side of \eqref{1} is positive and $a_0\zeta_0-\lim_{x\uparrow1}A(x)=c$, which means that $s=a_0+a_0\zeta_0-c<\infty$. Here we used the fact that if $\lim_{x\uparrow1}A(x)$ exists and $a_n\geq0$, then $\sum_{n=1}^{\infty}a_n=\lim_{x\uparrow1}A(x)$, that  in particular follows from Lemma \ref{lem1} for $\gamma=0$.
\end{proof}

The above proof of Theorem \ref{thm1} enables us to establish the following important property.

\begin{prop}\label{prop1}
Suppose that $a_0\zeta_0>\sum_{n=1}^{\infty}a_n$, and the sequence $b_n=a_{n+1}/a_n$, $n\geq0$, is strictly increasing. Then the sequence $\zeta_n$, $n\geq0$, strictly increases.
\end{prop}

\begin{proof}
Write $a_n\zeta_{n}-a_{n+1}\zeta_{n+1}=a_{n+1}$. Then,
\begin{equation}\label{2}
\sum_{k=1}^{n}a_k=a_0\zeta_0-a_n\zeta_{n}.
\end{equation}
If $a_0\zeta_0>\sum_{k=1}^{\infty}a_k$, then $c=a_0\zeta_0-\sum_{k=1}^{\infty}a_k>0$. From \eqref{2} we have

\begin{equation}\label{3}
\zeta_n=\frac{1}{a_n}\left(c+\sum_{k=1}^{\infty}a_k-\sum_{k=1}^{n}a_k\right)=\frac{1}{a_n}\left(c+\sum_{k=n+1}^{\infty}a_k\right).
\end{equation}
It follows from \eqref{3} that the sequence $\zeta_n$ is increasing. Indeed, we have
\[
\zeta_{n+1}=\frac{1}{a_{n+1}}\left(c+\sum_{k=n+2}^{\infty}a_k\right)>\frac{1}{a_{n}}\left(c+\sum_{k=n+1}^{\infty}a_k\right)=\zeta_n.
\]
The last inequality is true, since the made assumption implies that $a_{n+1}<a_n$ for all $n\geq0$ (if $a_{n+1}\geq a_{n}$ for a certain $n=n_0$, then the inequality must satisfy for all $n\geq n_0$, and we arrive at a divergent series),  and
\[
\frac{a_N}{a_n}=b_{N-1}b_{N-2}\cdot\ldots\cdot b_n<b_Nb_{N-1}\cdot\ldots\cdot b_{n+1}=\frac{a_{N+1}}{a_{n+1}}
\]
for any $N>n$.
\end{proof}

If $a_0\zeta_0<\sum_{k=1}^{\infty}a_k$, and the sequence $b_n$, $n\geq1$, is increasing, then it follows from \eqref{1} or \eqref{2} that there is the index value $n=n^*$ for which we have $\zeta_{n^{*}}\geq\zeta_{n^{*}-1}$, but $\zeta_{n^{*}+1}<\zeta_{n^{*}}$.

So, the idea of the search procedure is to find the value $\zeta_0$ such that $a_0\zeta_0$ would be close enough to $\sum_{k=1}^{\infty}a_k$.
The idea of the test procedure is to check whether the chosen value of $\zeta_0$ is given such that the sum of the series (or more often the remainder of the series) is less than given $\epsilon$.

\section{The test procedure}\label{S3}
Let \eqref{0} be a remainder of the series. The test procedure is aimed to answer the following question: whether $s-a_{n_0}<\epsilon$. So, setting $\zeta_{n_0}a_{n_0}=\epsilon$, we are to check whether the sequence $\zeta_n$, $n\geq n_0$, is increasing.

For the numerical illustration we consider the series $\sum_{n=1}^{\infty}n^{-3/2}\log(n+1)$. Some partial sums of this series, $S_n=\sum_{i=1}^{n}i^{-3/2}\log(i+1)$, are given in Table \ref{tab1}.
\begin{table}
    \begin{center}
            \caption{Some values of partial sums $S_n$ for the series $\sum_{n=1}^{\infty}n^{-3/2}\log(n+1)$.}\label{tab1}
        \begin{tabular}{|c||c|c|c|c|c||}\hline
$n$ & $5,000$  &  $10,000$ & $20,000$ & $50,000$ & $100,000$\\
\hline
$S_n$&$4.619697$   &  $4.692955$   & $4.748819$   & $4.802495$   & $4.831695$\\
       \hline
        \end{tabular}
    \end{center}
\end{table}
Note that for the aforementioned series, the monotonicity condition $b_n<b_{n+1}$, $n=1,2,\dots$, is satisfied.

 Using Theorem \ref{thm1} let us solve the following  problem. Take $n_0=10,001$ and $\epsilon=0.1$. Check whether $\sum_{n=10,001}^{\infty}n^{-3/2}\log(n+1)<0.1$.

To solve this problem, take $\zeta_{10,000}a_{10,000}=0.1$ and check whether the sequence $\zeta_{n}$, $n\geq10,000$, is increasing. In our analysis, we can check the values $\zeta_n$ for a fixed number of iterations only, say for $10,000\leq n\leq 59,999$.
Note, that because of this restriction, our analysis can wrongly show that all the obtained values $\zeta_n$ are indeed in the increasing order, while in fact the behaviour of $\zeta_n$ can be changed out of the horizon of $50,000$ iterations. In that case we may accept a wrong hypothesis and arrive at the mistaken result.

In our case the starting value is $\zeta_{10,000}=10,857.244172$. Then using the recurrence relation $\zeta_{n+1}=\zeta_na_n/a_{n+1}-1$ we find that  for $n=10,000, \dots, \ 17,804$ the sequence $\zeta_{n}$ follows in an increasing order, and then after $n=17,804$ it decreases. In Table \ref{tab0}, we provide some relevant values of $\zeta_n$ that indicate the behaviour of $\zeta_n$ prior the indicated number $n$. Thus, the solution to this problem yields the negative answer after less than $8,000$ steps of the recursion, that is many less than the maximum number of steps of the above convention.
\begin{table}
    \begin{center}
            \caption{Some values of $\zeta_n$ with the starting value $\zeta_{10,000}=10,857.244172$.}\label{tab0}
        \begin{tabular}{|c||c|c|c|c||}\hline
$n$  &  $17,802$ & $17,803$ & $17,804$ & $17,805$\\
\hline
$\zeta_n$  &  $12,736.509420$   & $12,736.509515$   & $12,736.509554$   & $12,736.509537$\\
       \hline
        \end{tabular}
    \end{center}
\end{table}
Following Table \ref{tab1}, $S_{10,000}=4.692955$ and $S_{10,000}+0.1=4.792955$. The last value is closer to $S_{50,000}$. More accurately, $S_{41,363}=4.792955$.

Let us now consider the same example with $\epsilon=0.15$. That is, we would like to check whether $\sum_{n=10,001}^{\infty}n^{-3/2}\log(n+1)<0.15$. In this case, $\zeta_{10,000}=16,285.866259$, and from the aforementioned recurrence relation we find that all the values $\zeta_n$, $10,000\leq n\leq 59,999$, follow an increasing order. In Table \ref{tab01}, we provide the four last values of $\zeta_n$. According to the numerical results obtained, we arrive at the positive answer to our hypothesis.
\begin{table}
    \begin{center}
            \caption{Some values of $\zeta_n$ with the starting value $\zeta_{10,000}=16,285.866259$.}\label{tab01}
        \begin{tabular}{|c||c|c|c|c||}\hline
$n$  &  $59,996$ & $59,997$ & $59,998$ & $59,999$\\
\hline
$\zeta_n$  &  $42,691.061392$   & $42,691.064068$   & $42,691.066728$   & $42,691.069372$\\
       \hline
        \end{tabular}
    \end{center}
\end{table}
Is the made conclusion correct? The obtained value is greater than that indicated in Table \ref{tab1} for $S_{100,000}$. So, we indeed can believe that our solution is true.

Let us now re-check whether our conclusion is true. Take $n=100,000$. Then from Table \ref{tab1} we have $\epsilon=4.842955-4.831695=0.011260$. The associated value of $\zeta_{100,000}$ is $\zeta_{100,000}=30,928.034437$. Now a new recalculation shows that our previous conclusion was wrong. Starting with $\zeta_{100,000}=30,928.034437$ one can observe that the sequence $\zeta_n$ is not increasing, just decreasing. Even the second value $\zeta_{100,001}=30,927.471495$ is less than the first (original) value $\zeta_{100,000}$. Summing up the series terms that are out of Table \ref{tab1}, we find $S_{139,230}=4.842955$. These two test calculations show a massive difference between the first test given for $n=10,000$ and $\epsilon=0.15$ and the second one given for $n=100,000$ and $\epsilon=0.011260$. In the first case the $50,000$ steps of iterations were insufficient to arrive at true conclusion, while in the second case the only single iteration provided a true conclusion. Indeed, in the second case the information on the partial sum $S_{100,000}$ is more complete about the series, and the smaller value of $\epsilon$ compared to its originally defined value enables us to provide a more exact verification of the test. The last conclusion follows directly from \eqref{2}. If $a_0\zeta_0<\sum_{n=1}^{\infty}a_n$, then there exists $n^*$ such that the partial sum $S_{n^*}>a_0\zeta_0$ and $\zeta_{n^*}$ must be negative. Prior becoming negative, the sequence $\zeta_n$ that starts from the positive $\zeta_0$ must decrease. So, if $n<n^*$ is close to $n^*$, then for the remainder of the series the sequence $\zeta_n$ will decrease,
and this effect has just been obtained numerically.

\section{The search procedure}\label{S4} The main idea of the search method is a sequential evaluation of the sum of series or its remainder. 

\subsection{Search algorithms}

\noindent
\textit{The algorithm of step-forward search}

\begin{enumerate}
\item [(i)] \textit{Initial step.} For some $N$ find the partial sum of the series $S_N=\sum_{n=n_0}^{N}a_n$.
\item [(ii)] \textit{Test search step}. For a given $\epsilon$ test whether the remainder of the series $\sum_{n=N+1}^{\infty}a_n$ is less than $\epsilon$.
\item [(iii)] If the answer in (ii) is negative, then find a new value of $S_{N^*}$, where
\[
N^*=\min\left\{m:\sum_{n=n_0}^{m}a_n\geq S_N+\epsilon\right\}=\min\{m: S_m\geq S_N+\epsilon\},
\]
set $S_N=S_{N^*}$ and repeat (i) and (ii).
\item [(iv)] If the answer in (ii) is positive, the procedure is terminated.
\end{enumerate}

Using this method assumes that the possible number of iterations at a test search step can be large. Then wrong decision at final step in the series of the test search steps can be made with negligibly small likelihood. The maximum number of iterations in a step is set to $10^9$.

\medskip
\noindent
\textit{The algorithm of modified step-forward search}

\begin{enumerate}
\item [(i)] \textit{Initial step.} For some $N$ find the partial sum of the series $S_N=\sum_{n=n_0}^{N}a_n$.
\item [(ii)] \textit{Test search step}. For a given $\epsilon$ test whether the remainder of the series $\sum_{n=N+1}^{\infty}a_n$ is less than $\epsilon$.
\item [(iii)] If the number of iterations is less than a specified value $M$ before the negative answer is obtained, then we find a new value of $S_{N^*}$, where
\[
N^*=\min\left\{m:\sum_{n=n_0}^{m}a_n\geq S_N+\epsilon\right\}=\min\{m: S_m\geq S_N+\epsilon\},
\]
and repeat (i) and (ii).
\item [(iv)] If the number of iteration reaches $M$, then the procedure of test search is interrupted.
\item [(v)] \textit{New test search step}. The test search is resumed with the new parameter $\epsilon^*=\epsilon/K$.
\item [(vi)] If the answer in (v) is negative, then find a new value of $S_{N^*}$, where
\[
N^*=\min\left\{m:\sum_{n=n_0}^{m}a_n\geq S_N+\epsilon^*\right\}=\min\{m: S_m\geq S_N+\epsilon^*\},
\]
set $S_N=S_{N^*}$ and repeat (v) and (vi).
\item [(vii)] If the answer in (v) is positive, the procedure is terminated.
\end{enumerate}

\noindent
\medskip
\textit{Remarks}
\begin{enumerate}
\item [1.] In general, $M\geq2$. In the numerical study in Section \ref{S4.3} we set $M=2$ that seems to be the best setting in the general situation.
\item [2.] The most convenient setting for $K$ is $K=10$.
\item [3.] The presented algorithm can be further modified. For instance, after step (v) we can check the number of iterations again similarly to that it is given in step (iii). If it is less than $M$, then we find $S_N$ as indicated in (vi). Otherwise the procedure is interrupted and then resumed with the new parameter $\epsilon^{**}=\epsilon^*/K$ and so on.
\item [4.] Following the above three remarks, the total number of iterations in order to reach the required accuracy of the series can be made relatively small.
\end{enumerate}

\subsection{Numerical study}\label{S4.3} For the numerical study we consider the same series
\begin{equation}\label{4}
\sum_{n=1}^{\infty} \frac{\log(n+1)}{n\sqrt{n}}
\end{equation}
that was considered in Section \ref{S3} as well as the series
\begin{equation}\label{5}
\sum_{n=1}^{\infty} \frac{\log(n+1)}{n\sqrt{n}\sqrt[4]{n}}
\end{equation}
that converges with the higher rate compared to the series given by \eqref{4} and hence can be provided with higher accuracy. For series \eqref{4} we provide our experiments with $\epsilon=0.01$ taking the initial partial sum $S_{100,000}=4.831695$ (see Table \ref{tab1}). For series \eqref{5} we use $\epsilon=0.0001$ starting with the initial partial sum $S_{1,000,000}=2.625626$.

\subsubsection{Step-forward search} With $\epsilon=0.01$ the step-by-step results for series \eqref{4} are given in Table \ref{tab2} and for series \eqref{5} in Table \ref{tab21}.
\begin{table}
    \begin{center}
            \caption{Numerical study of the series $\sum_{n=1}^{\infty} \log(n+1)n^{-3/2}$ by step-forward search with $\epsilon=0.01$}\label{tab2}
        \begin{tabular}{c|c|c|c}\hline
Step number  & Number of iterations & $n$       & $S_n$\\
                & in the step          &           &      \\
\hline\hline
 0              & N/A                 & $100,000$ &$4.831695$\\
 1              & $1$                  & $133,854$ &$4.841695$\\
 2              & $1$                  & $186,526$ &$4.851695$\\
 3              & $1$                  & $274,211$ &$4.861695$\\
 4              & $1$                  & $434,474$ &$4.871695$\\
 5              & $1$                  & $789,816$ &$4.881695$\\
 6              & $413,543$            &$1,702,013$&$4.891695$\\
 7              & $6,248,811$          &$5,401,971$&$4.901695$\\
 8              & $333,412,235$        &$62,126,060$&$4.911695$\\
 9              & $10^9$     & $10^9$       &$4.915721$\\
       \hline
        \end{tabular}
    \end{center}
\end{table}

\begin{table}
    \begin{center}
            \caption{Numerical study of the series $\sum_{n=1}^{\infty} \log(n+1)n^{-7/4}$ by step-forward search with $\epsilon=0.0001$}\label{tab21}
        \begin{tabular}{c|c|c|c}\hline
Step number  & Number of iterations & $n$       & $S_n$\\
                & in the step          &           &      \\
\hline\hline
 0              & N/A                 & $1,000,000$ &$2.625626$\\
 1              & $1$                  & $1,282,406$ &$2.625726$\\
 2              & $1$                  & $1,730,125$ &$2.625826$\\
 3              & $1$                  & $2,251,124$ &$2.625926$\\
 4              & $1$                  & $4,189,924$ &$2.626026$\\
 5              & $96,723$             & $9,190,084$ &$2.626126$\\
 6              & $6,975,835$          &$57,584,662$&$2.626226$\\
 7              & $10^{9}$          &$10^{10}$&$2.626263$\\
       \hline
        \end{tabular}
    \end{center}
\end{table}

It is seen from Table \ref{tab2} that by the only 9 steps we arrive at the result giving us the approximate value of the series $4.915318$. The result with two decimal places for the sum of series is $4.92$ that achieved with approximately $10^9$ terms.  Note also that in the first 5 steps, there is only a single iteration, while when we arrive closer to the end, the number of iterations within the step essentially increases. The essential grows of the number of iterations is seen in steps $6$, $7$ and $8$, while in step $9$ the number of iterations reaches the established limit of $10^9$.

For the series given by \eqref{5}, the required result is achieved by 7 steps. From Table \ref{tab21} we see that the number iterations at step $5$ is $96,723$ and the number of iterations in step $6$ is $57,584,662$. The $7$th step is final, and the resulting sum of the series is approximately $2.626263$. The result with four decimal places for this series is $2.6263$. It is achieved after summing up approximately $10^{9}$ terms.

\subsubsection{Modified step-forward search} Numerical study with modified step-forward search is provided with same $\epsilon$ and  $M=2$. This means that for series \eqref{4}  all calculation starting from step $6$ are to be provided with parameter $\epsilon^*=0.001$. Step-by-step results that correspond to steps $8$ and $9$ in step-forward method of Table \ref{tab2} are now shown in Table \ref{tab3}. For series \eqref{5}, all calculation starting from step $5$ are to be provided with parameter $\epsilon^*=0.00001$. Step-by-step results that correspond to steps $6$ and $7$ in step-forward method of Table \ref{tab21} are now shown in Table \ref{tab31}.

\begin{table}
    \begin{center}
            \caption{Numerical study of the series $\sum_{n=1}^{\infty} \log(n+1)n^{-3/2}$ by modified step-forward search with $\epsilon=0.01$ and $\epsilon^*=0.001$} \label{tab3}
        \begin{tabular}{c|c|c|c}\hline
Step number  & Number of iterations & $n$       & $S_n$\\
                & in the step          &           &      \\
\hline\hline
 25              & $1$                 & $5,401,971$ &$4.901695$\\
 26              & $1$                  & $6,307,961$ &$4.902695$\\
 27              & $1$                  & $7,449,235$ &$4.903695$\\
 28              & $1$                  & $8,912,398$ &$4.904695$\\
 29              & $1$                  & $10,827,113$ &$4.905695$\\
 30             & $1$                  & $13,394,222$ &$4.906695$\\
 31              & $1$                  &$16,937,648$&$4.907695$\\
 32              & $1$                  &$22,005,935$&$4.908695$\\
 33              & $1$                  &$29,585,579$&$4.909695$\\
 34              & $1$               &$41,590,939$&$4.910695$\\
 35              & $1$               &$62,126,060$&$4.911695$\\
 36              & $1$                &$101,277,959$&$4.912695$\\
 37              & $1$                &$189,350,834$&$4.913695$\\
 38              & $8,546,857$        &$453,021,228$&$4.914695$\\
 39              & $10^9$                &$10^9$ &$4.915721$\\
       \hline
        \end{tabular}
    \end{center}
\end{table}

\begin{table}
    \begin{center}
            \caption{Numerical study of the series $\sum_{n=1}^{\infty} \log(n+1)n^{-7/4}$ by modified step-forward search with $\epsilon=0.0001$ and $\epsilon^*=0.00001$} \label{tab31}
        \begin{tabular}{c|c|c|c}\hline
Step number  & Number of iterations & $n$       & $S_n$\\
                & in the step          &           &      \\
\hline\hline
 14              & $1$                 & $9,190,084$&$2.626126$\\
 15              & $1$                  & $10,238,361$ &$2.626136$\\
 16              & $1$                  & $11,505,615$ &$2.626146$\\
 17              & $1$                  & $13,062,296$ &$2.626156$\\
 18              & $1$                  & $15,011,116$ &$2.626166$\\
 19             & $1$                  & $17,507,220$ &$2.626176$\\
 20              & $1$                  &$20,795,233$&$2.626186$\\
 21              & $1$                  &$25,281,898$&$2.626196$\\
 22              & $1$                  &$31,690,710$&$2.626206$\\
 23              & $1$               &$41,428,077$&$2.626216$\\
 24              & $1$               &$57,584,662$&$2.626226$\\
 25              & $1$                &$88,308,941$&$2.626236$\\
 26              & $1$                &$162,735,728$&$2.626246$\\
 27              & $47,811,731$        &$482,815,421$&$2.626256$\\
 28              & $856,114,482$                &\text{over} \  $10^{10}$ &$2.626266$\\
       \hline
        \end{tabular}
    \end{center}
\end{table}

\section{Concluding remark}\label{S5}

In the present paper we suggested a new method of estimating the sum of positive convergent series. The numerical procedures based on this method shows their effectiveness for a wide class of series.
The assumption that the sequence $b_n=a_{n+1}/a_n$ is strictly increasing is quite natural. The reasonable questions are how important this assumption is, what if it is not satisfied.

Under the made assumption, the test procedure reduces to find the first value $\zeta_n$ in the sequence that less then previous one. If such value is found, then the hypothesis is rejected. Modified step-forward search improves the construction and made the search procedure quicker.

If this assumption about the sequence $b_n$ is not satisfied, then the test procedure becomes much longer, since in that case we are required a many larger number of iterations to find the first negative value of $\zeta_n$ in the sequence, and only then we reject the proposed hypothesis. The larger number of operations affects negatively on the performance and makes impossible to use the modification of the search method that is used in the case when the aforementioned assumption on the sequence $b_n$ is satisfied.

\section*{Disclosure statement}

No conflict of interests was reported by the author.
\bibliographystyle{amsplain}

\end{document}